\documentclass[twoside]{article}

\usepackage{amsmath}
\usepackage{amssymb}
\usepackage{amsthm}
\usepackage{amsfonts}
\usepackage{fancyhdr}
\usepackage{hyperref}
\usepackage{color}
\usepackage{enumitem}
\usepackage{todonotes}
\usepackage{mathabx} 
\usepackage{changes}
\usepackage{empheq}
 \usepackage{cite}

\usepackage{mathtools}
 \mathtoolsset{showonlyrefs}
\theoremstyle{definition}
\newtheorem{definition}{Definition}[section]

\theoremstyle{plain}

\newtheorem{lem}[definition]{Lemma}
\newtheorem{teo}[definition]{Theorem}
\newtheorem{Propos}[definition]{Proposition}
\newtheorem{Cor}[definition]{Corollary}

\theoremstyle{remark}
\newtheorem{remark}{Remark}[section]
 
\newcommand{\orlnor}{\|_{L^{\phi}}}

\newcommand{\lphi}{L^{\phi}}
\newcommand{\lpsi}{L^{\phi^{\star}}}
\newcommand{\wphi}{W^{1}\lphi}
\newcommand{\wphit}{W^{1}\lphi_T}
\newcommand{\wphittilde}{\widetilde{W}^1\lphi_T}
\newcommand{\sobnor}{\|_{W^{1}\lphi}}
\newcommand{\sobnort}{\|_{W^{1}\lphi_T}}
\newcommand{\rr}{\mathbb{R}}

\makeatletter
\newcommand{\labitem}[2]{%
\def\@itemlabel{\textnormal{\textbf{(#1)}}}
\item
\def\@currentlabel{\textnormal{\textbf{(#1)}}}\label{#2}}
\makeatother
\makeatletter
\def\namedlabel#1#2{\begingroup
    #2%
    \def\@currentlabel{#2}%
    \phantomsection\label{#1}\endgroup
}
\makeatother

\title{On the existence of periodic solutions for $\phi$-Laplacian inclusion systems\thanks{   The research has been supported by SECyT-UNRC, grant C561 and FCEyN-UNLPam, grant PI 77 M.
}}
\author{
Stefania M. Demaria and  Fernando D. Mazzone \\
Dpto. de Matem\'atica, Facultad de Ciencias Exactas, F\'{\i}sico-Qu\'{\i}micas y Naturales\\
Universidad Nacional de R\'{i}o Cuarto\\
(5800) R\'{\i}o Cuarto, C\'ordoba, Argentina,\\
\url{sdemaria@exa.unrc.edu.ar},
\url{fmazzone@exa.unrc.edu.ar} \\
}

\date{}

\begin{document}
\maketitle
\begin{abstract}
We apply the direct method of the calculus of variations to prove existence of periodic solutions for differential inclusion systems involving an anisotropic $\phi$-Laplacian operator.
 \end{abstract}

 \noindent\textit{keywords:} differential, inclusions, periodic, variational, methods.

\section{Introduction}

Let $\phi:\mathbb{R}^n \to [0,+\infty)$ be a convex and differentiable function, and let $T > 0$. We consider a function $F : [0,T] \times \mathbb{R}^n \to \mathbb{R}^n$, which is locally Lipschitz in $x$ for almost every $t \in [0,T]$. The Clarke generalized gradient of $F$ with respect to $x$ is denoted by $\partial F$ (see \cite{clarke1990optimization, clarke2013functional}).

The aim of this paper is to establish the existence of solutions to the following differential inclusion problem with periodic boundary conditions:

\begin{equation}\label{eq_1}\left\{
\begin{array}{ll}
   \frac{d}{dt} \nabla\phi(u'(t))\in \partial F(t,u(t)) \quad \hbox{a.e.}\ t \in (0,T)\\
    u(0)-u(T)=u'(0)-u'(T)=0,
\end{array}\right.\tag{$P$}
\end{equation}

We obtain solutions to problem \eqref{eq_1} using the direct method in the calculus of variations. To this end, we introduce the following functional, also referred as \emph{action integral}:
\begin{equation}\label{funcional}
    \mathcal{I}(u):=\int_{0}^{T}\phi(\dot u(t))dt+\int_{0}^{T}F(t, u(t))dt.\tag{I}
\end{equation}
The objective is to show that, under suitable assumptions, the functional $\mathcal{I}$ attains a minimum in the Sobolev–Orlicz vector space $W^{1,\phi}_T([0,T], \mathbb{R}^n):=W^{1,\phi}([0,T], \mathbb{R}^n)\cap \{u \mid u(0)=u(T)\}$, and that the minimizers of $\mathcal{I}$ are solutions to problem \eqref{eq_1}.

The case of problem \eqref{eq_1} with a Fréchet differentiable function $F$ and  equations, instead of an inclusions, has been extensively studied in the literature.
 In   \cite{mawhin2010critical},  J. Mawhin and M. Willem applied   various methods: direct, dual, saddle points,  minimax, topological degree to this problem when  $\phi(x)=|x|^2/2$. This theory was extended in several articles,  see for example \cite{tang1995periodic,tang1998periodic,tang2001periodic,wu1999periodic,zhao2004periodic}. The \emph{$p$-laplacian operador } ($\phi(x)=|y|^p/p$,  $1<p<\infty$) were considered in  \cite{Tang2010,tian2007periodic}.  If

\begin{equation}\label{eq:phip_1p_2}\phi=\phi_{p_1,p_2}(y_1,y_2):=\frac{|y_1|^{p_1}}{p_1}+\frac{|y_2|^{p_2}}{p_2},\end{equation}
then \eqref{eq_1}  becomes in a periodic boundary problem for a  $(p_1,p_2)$-laplacian system.  This class of problem was studied in
\cite{li2014periodic,pasca2010periodic,pacsca2010some,pasca2011some,pasca2016periodic,yang2012periodic,yang2013existence}.  The question of finding periodic solutions for systems involving general $\phi$-laplacian operators was studied in \cite{acinas2015some-2,Mazzone2019,acinas2022,acinas2019clarke,Acinas2017}, while the case of differential inclusions was addressed in \cite{Pasca2011, pacsca2007periodic,pacsca2008subharmonic,Ning2015}. Hence \eqref{eq_1}  contains several problems that have been considered by many authors in the past. Moreover, our results improve, in many cases, the existing ones in the literature, even when restricted to the particular case of $(p_1,p_2)$-Laplacian systems. The latter is discussed at greater length in the  Section \ref{sec:applications}. For these reasons, we believe that anisotropic Orlicz-Sobolev spaces offer a suitable framework for unifying many known results.

\section{Main results}
We aim to formulate some of our results within the more general framework of Banach spaces, as we believe this may prove useful for future applications to partial differential equations.

 Throughout this paper, we denote by $(X,\|\cdot\|_X)$ an arbitrary Banach space. In the particular case where $X = \mathbb{R}^n$, we assume that $\|\cdot\|_X$ is the Euclidean norm and  denote it simply by  $|\cdot|$. 
 
Let us start by giving a brief introduction to Orlicz and Orlicz–Sobolev spaces of Banach space-valued functions. First, we introduce  $G$-functions.  References for  these topics include \cite{orliczvectorial2005,Mazzone2019,trudinger1974imbedding,acinas2019clarke,chamra2017anisotropic}.

  Following to Trudinger (see \cite{trudinger1974imbedding}) we say that  $\phi: X \to [0,+\infty]$   is a \emph{$G$-function} if it satifies that $\phi$ is convex, $\lim_{|x|\to\infty}\phi(x)=\infty$, $\phi(0)=0$, $\phi(-x)=\phi(x)$, $\phi$ is bounded in some  neighborhood of $0$ and $\phi$ is lower semicontinuous.

Associated to $\phi$ we have the \emph{complementary function} $\phi^{\star}$ which is defined at $y\in X$ as
\begin{equation}\label{eq:conjugada}
 \phi^{\star}(\xi)=\sup\limits_{x\in X} \langle \xi, x\rangle-\phi(x),
\end{equation}
where $ \langle \cdot, \cdot\rangle : X^\star\times X\to\mathbb{R}$ is the canonical pairing between $X$ and its dual space.

\begin{definition}\label{def:n-function}
A $G$-function $\phi:X \to [0, \infty)$ is called an \emph{N-function}  if $\phi(x) = 0$ if and only if $x = 0$, and it satisfies
\[
\lim_{\|x\|_X \to 0} \frac{\phi(x)}{\|x\|_X} = 0 \text{ and } \lim_{\|x\|_X \to \infty} \frac{\phi(x)}{\|x\|_X} = \infty.
\]
\end{definition}

We say that a $G$-function $\phi:X\rightarrow [0,+\infty)$ satisfies the  \emph{$\Delta_2$-condition} and we denote  $\phi \in \Delta_2$,
if there exists a  constant $C>0$  such that
\begin{equation}\label{delta2defi}\phi(2x)\leq C \phi(x)+1,\quad x\in X.
\end{equation}
Note that this definition is equivalent to the classic one, 
i.e. there exist $r_0, C > 0$  with $\phi(2x) \leq C\phi(x)$ for $\|x\| > r_0$.  When the inequality $\phi(2x)\leq C \phi(x)$ is satisfied for every $x$, we will say that $\phi$  satisfies the \emph{$\Delta_2$-condition globally}.  We denote this fact by $\phi\in \Delta_2^G$.

As in  \cite{rao1991theory}, we write $\phi_1\llcurly\phi_2$ if for every $k>0$ there exists  $C>0$ such that
\begin{equation}\label{eq:orden} \phi_1(x)\leq C+\phi_2(kx),\quad x\in X.\end{equation}

\begin{remark} Again, this definition is equivalent to say that for every $k>0$ there exists $R>0$ such that $\phi_1(x)\leq\phi_2(kx)$ for every $\|x\|_X>R$. 
\end{remark}

In this article, unless otherwise specified, the measurability of functions or sets will be understood with respect to the Lebesgue $\sigma$-algebra on the interval $[0, T]$. We say that $u:[0,T]\to X$ is \emph{Bochner measurable} if $u$ is a.e. limit  of a sequence of simple functions.  It is easy to see that if $u$ is Bochner measurable and $\phi$ is   a $G$-function then $\phi(u)$ is measurable. We define the \emph{modular function} by
\[\rho_{\phi}(u):= \int_0^T \phi(u)\, dt.\]

The \emph{Orlicz space} $\lphi=L^{\phi}\left([0,T],X\right)$ is given by
\begin{equation}\label{espacioOrlicz}
\lphi:=\big\{ u\big| u \text{ Bochner measurable and } \exists \lambda>0: \rho_{\phi}(\lambda u) < \infty   \big\}.
\end{equation}

  The Orlicz space $\lphi$ equipped with the \emph{Luxemburg norm}
\[
\|  u  \orlnor:=\inf \bigg\{ \lambda\bigg| \rho_{\phi}\bigg(\frac{v}{\lambda}\bigg) \,dt\leq 1\bigg\},
\]
is a Banach space.

We define the \emph{Orlicz-Sobolev space} $\wphi=\wphi\left([0,T],X\right)$ by
\[\wphi:=\big\{u| u \text{ is absolutely continuous, }  u'\in \lphi\big\}.\]
See \cite{barbu1976nonlinear} for the definition and properties of $X$-valued  absolutely continuous functions.
The space $\wphi$ is a Banach space when it is equipped with the norm
\begin{equation}\label{def-norma-orlicz-sob}
\|  u  \|_{\wphi}= \|  u  \|_{\lphi} + \|u'\orlnor.
\end{equation}
The subspace $\wphit$ of $\wphi$ is defined by
\[\wphit:=\wphi\cap \big\{u| u(0)=u(T)\big\}.\]
Note that $\wphit$ is a closed subspace of $\wphi$.

Let $f:X\to \rr$ a locally Lipschitz function and   $x, v \in X$,  following \cite{clarke1990optimization, clarke2013functional}, we define the \emph{generalized directional derivative} of $f$  at $x$ in the direction $v$  as

$$
f^0(x ; v)=\limsup _{y \rightarrow x, \lambda \searrow 0} \frac{f(y+\lambda v)-f(y)}{\lambda}
$$
and we denote by

$$
\partial f(x)=\left\{\xi \in X^\star: f^0(x ; v) \geqslant\left\langle \xi, v\right\rangle, \text { for all } v \in X\right\}
$$
the \emph{Clarke generalized gradient} of $f$ at $x$.

We recall that $f$ is called \emph{regular} when $f^0(x;v)$ agree with the usual directional derivative (see \cite{clarke1990optimization,clarke2013functional}).

Our main results consist of three existence theorems for solutions, each of which is based on different hypotheses on the function $F$; however, the following assumptions are common to all three theorems:

\begin{enumerate}
    \labitem{H1}{hip:H1} $\phi:\mathbb{R}^n \to [0,+\infty)$ is a strictly convex $G$-function such that both $\phi$ and its conjugate $\phi^\star$ satisfy the $\Delta_2$-condition.
    \labitem{H2}{hip:H2} For a.e.\ $t \in [0,T]$, the function $F$ is regular and locally Lipschitz with respect to $x$.
    \labitem{H3}{condacotación} There exist a function $b \in L^1([0,T];[0,\infty))$ and a function $c:\mathbb{R}^n \rightarrow \mathbb{R}$, with $a$ bounded on bounded subsets, such that for every $x \in \mathbb{R}^n$, for a.e. $t \in [0,T]$ and for every  $\xi_t \in \partial F(t,x)$,
        \begin{equation}\label{eq:condacotación}
            |F(t,x)| + |\xi_t| \leq c(x)b(t).
        \end{equation}
\end{enumerate}

 \begin{teo}\label{coercotasubgrad} Suppose that \ref{hip:H1}--\ref{condacotación}  hold, along with the following assumptions:
\begin{enumerate}

    \labitem{H4}{condB} There exist a function $d \in L^1([0,T],\mathbb{R})$ and a $G$-function $\phi_0 \llcurly \phi$ such that for every $\xi_t \in \partial F(t,x)$,
    \begin{equation}\label{eq:condition_grad}
        \phi^\star\left(\frac{\xi_t}{d(t)}\right) \leq \phi_0(x) + 1.
    \end{equation}
    
    \labitem{H5}{condcoermedia1} \textit{Coercivity condition:}
    \begin{equation}\label{eq:I->infinito}
        \lim_{|x| \to \infty} \frac{1}{\phi_0(2x)} \int_0^T F(t,x)\,dt = \infty.
    \end{equation}
\end{enumerate}
Then the functional $\mathcal{I}$ admits a minimizer $u$, which solves problem~\eqref{eq_1}.
\end{teo}

We also establish the existence of solutions under a different set of assumptions:

\begin{teo}\label{coercuasi}Let \ref{hip:H1}--\ref{condacotación} be satisfied, and consider the following additional conditions:
\begin{enumerate}
    \labitem{H6}{cuasisub} \textit{Quasi-subadditivity:} There exist constants $\lambda, \mu > 0$ such that
    \[
    F(t, \lambda(x+y)) \leq \mu \big(F(t,x) + F(t,y)\big).
    \]
    \labitem{H7}{phiceromenor} There exist a $G$-function $\phi_0 \llcurly \phi$ and a function $b \in L^1([0,T],\mathbb{R})$ such that
    \[
    F(t,x) \leq b(t)\big(\phi_0(x) + 1\big).
    \]
    \labitem{H8}{coerc2} \textit{Coercivity condition:}
    \begin{equation}\label{condcoermedia}
    \lim_{|x| \to \infty} \int_0^T F(t,x)\,dt = \infty.
    \end{equation}
\end{enumerate}
Then the functional $\mathcal{I}$ admits a minimizer $u$, which solves problem \eqref{eq_1}.
\end{teo}

Before formulating our third theorem, let us introduce the following definition.

\begin{definition}\label{def:bo}
We say that the continuous function  $f:\mathbb{R}^n \to \mathbb{R}$ belongs to the space  $\mathcal{BO}$ if it satisfies that
\[[f]_{\mathcal{BO}}:=\sup_{x, y \in\mathbb{R}^n}\frac{|f(x)-f(y)|}{1+|x-y|}<\infty.\]
It is immediate see that $[\cdot]_{\mathcal{BO}}$ is a semi-norm over  $\mathcal{BO}$.
\end{definition}

\begin{teo}\label{coerBO}Assume that conditions \ref{hip:H1}–\ref{condacotación} and \ref{coerc2} hold, that $\phi$ is an $N$‑function, and that the following additional assumptions are satisfied:

\begin{enumerate}
     \labitem{H9}{it:BO} \label{inec38} There exists $ b \in L^1([0,T], \mathbb{R})$ such that
    $[F]_{\mathcal{BO}}\leq b(t)$,  where the seminorm is taken over $x$ for each fixed $t$.
\end{enumerate}
Then  the functional $\mathcal{I}$ has a minimum $u$, which solves problem \eqref{eq_1}.

\end{teo}

\section{Preliminaries}\label{Preliminares}

We would like to point out that throughout this article, the letter $C$ will be used to denote constants, and its value may vary from line to line, even within a chain of inequalities.

Given a Banach space $X$ we denote by $\langle \cdot,\cdot\rangle$ the usual bilinear quadratic form defined on $ X^\star\times X$.

Let us start by recalling a few simple facts about $G$-functions.  We suggest \cite{Mazzone2019,2002applications,donaldson1971orlicz,trudinger1974imbedding} as references for the results we list below.

\begin{Propos}\label{propos:Gfuncion_elementales} If $\phi:X\to [0,\infty)$ is a $G$-function then

\begin{enumerate}
 \item     $\phi(\lambda x)\leq \lambda\phi(x)$, for every  $\lambda\in[0,1],x\in X$;
 \item $\phi(\lambda_1 x)\leq\phi(\lambda_2 x)$, when  $0<|\lambda_1|\leq |\lambda_2| $,
 \item If $\phi$ is Frechet differentiable then $\phi^{\star}\left(\phi'(x) \right)\leq \langle \phi'(x),x\rangle\leq \phi(2x)$.
  \item $\langle \xi , x\rangle \leq  \phi(x)+\phi^{\star}(\xi)$ (Fenchel inequality).
\item If $\phi^{\star}\in\Delta_2$ then $\phi$ satisfies the \emph{$\nabla_2$-condition}, i.e.  for every $0<r<1$ there exist $l=l(r)>0$ and $C=C(r)>0$ such that
\begin{equation}\label{eq:nabla2}
  \phi(x)\leq \frac{r}{l}\phi(l x)+C,\quad x\in X.
\end{equation}

\end{enumerate}

\end{Propos}

\begin{remark}
 It is easy to see that \eqref{eq:nabla2}  is equivalent  to  assume that there exists $R>0$ such that   \eqref{eq:nabla2} holds  with $C=0$,  $r=1/2$ and for $|x|>R$.
\end{remark}
In what follows we give some properties of   Orlicz and Orlicz-Sobolev spaces.

 \begin{Propos}[\emph{H\"older's inequality}] If $u\in\lphi([0,T],X)$ and $v\in\lpsi([0,T],X^\star)$ then $\langle  v,u \rangle   \in L^1$ and
\begin{equation}\label{holder}
\int_0^T \langle v,u \rangle   dt\leq 2 \|u\orlnor\|v\|_{L^{\phi^{\star}}}.
\end{equation}
 \end{Propos}

Suppose $u\in\lphi([0,T],X)$ and consider $K:=\rho_{\phi}(u)+1\geq 1$. Then, from item 1 in Proposition \ref{propos:Gfuncion_elementales} we have $\rho_{\phi}(K^{-1}u)\leq K^{-1}\rho_{\phi}(u)\leq 1$. Therefore, we conclude
\begin{equation}\label{eq:amemiya}
 \|u\orlnor \leq \rho_{\phi}(u)+1.
\end{equation}

As is customary, we will use the decomposition $u=\overline{u}+\widetilde{u}$ for a function $u\in L^1([0,T],X)$,  where $\overline{u} =\frac1T\int_0^T u(t)\, dt$ and $\widetilde{u}=u-\overline{u}$. The last integral is taken in the Bochner sense (see \cite{papageorgiou2009handbook}). Then, one has
\[\wphit=\wphittilde\oplus X, \]
where  $X$  is identified with the set of constant functions and   \[\wphittilde=\{u \in \wphit: \overline{u}=0  \}.\]

We recall   (see Lemma 2.4 in \cite{Mazzone2019} and Theorem 4.4 in \cite{chamra2017anisotropic}) the following inequality.

\begin{lem}
\label{lem:inclusion orliczII} Let $\phi:X\to [0,+\infty)$ be a $G$-function,
and let $u\in\wphit\left([0,T],X\right)$. Then

\begin{equation}\label{eq:wirtinger}
  \phi\big(\tilde{u}(t)\big)\leq\frac{1}{T} \int_0^T \phi\big(Tu'(r)\big)\,dr.
\end{equation}

\end{lem}

  \begin{Propos}\label{prop:prop_espacios} Let $\phi$ be a $G$-function. Then:
  \begin{enumerate}
\item\label{it:1} There exists a positive constant   $C$ such that $\|u\|_{L^1}\leq C\|u\|_{L^\phi}$, for every $u\in L^\phi([0,T],X)$ and $\|u\|_{L^\phi}\leq C\|u\|_{L^\infty}$, for every $u\in L^\infty$. Briefly $L^\infty([0,T],X) \hookrightarrow L^\phi([0,T],X)\hookrightarrow L^1([0,T],X)$. Consequently\linebreak $\wphi([0,T],X) \hookrightarrow W^1L^1([0,T],X) $.

\item\label{it:1,5} (Anisotropic Poincaré-Wirtinger Inequality) For every $u\in\lphi([0,T],X)$ we have that $\|\tilde{u}\orlnor\leq T\|u'\orlnor$.  Even more $\|\tilde{u}\|_{L^\infty}\leq C\|u'\orlnor$.

  \item\label{it:2}  $\wphi([0,T],X) \hookrightarrow  C^0([0,T],X)$.

\item\label{it:4} $\|u\sobnort':=|\overline{u}|+\|u'\orlnor$ defines an equivalent norm to $\|\cdot\sobnor$ on \linebreak $\wphit([0,T],\rr^n)$.

\item\label{it:5} Every bounded sequence $\{u_n\}$ in  $\wphi([0,T],\rr^n)$  has an uniformly convergent subsequence.

\item\label{it:6}   If  $u_n\rightharpoonup u$  (as usual $\rightharpoonup $ denotes weak convergence)   in  $\wphi([0,T],\rr^n)$ then $u_n$ converges to $u$ uniformly.
     \end{enumerate}
  \end{Propos}
\begin{proof}All the results listed above are essentially known, as can be seen in the cited literature. For completeness, we include brief proofs of those results that we have not found explicitly stated in the references.

Applying inequality \eqref{eq:wirtinger} to the function $u/T\|u'\orlnor$ we get
\begin{equation}\label{eq:norm_inf_tilde}
 \phi\left(\frac{\tilde{u}}{T\|u'\orlnor}\right)\leq \frac{1}{T}\int_0^T\phi\left(\frac{u'}{\|u'\orlnor}\right)\leq T^{-1}.
\end{equation}
Integrating \eqref{eq:norm_inf_tilde} for $t$ in $[0,T]$ we obtain the first inequality in item   \ref{it:1,5}. Since $\phi(x)\to\infty$, when $|x|\to\infty$, we can find a constant $R>0$ such that $\phi(x)>1/T$ when $|x|>R$. Therefore, from \eqref{eq:norm_inf_tilde} we obtain $\|\tilde{u}\|_{L^\infty}\leq RT \|u'\orlnor$.

 Item \ref{it:2} is a inmediate consequence of item \ref{it:1} and    \cite[Th. 2.2]{buttazzo1998one}.

 For \ref{it:4}, note that by triangle inequality for integrals and by item \ref{it:1}, we have that
 $$
 |\overline{u}|\leq T^{-1}\|u\|_1 \leq C\|u\|_{W^1L^1}\leq C \|u\sobnor,
 $$
which constitutes one of the inequalities necessary to establish the equivalence. The other one is trivially inferred from the chain $\|\overline{u}\orlnor\leq C\|\overline{u}\|_{C^0}=C|\overline{u}|$.

Item \ref{it:5} is a consequence of item \ref{it:2} and the Arzela-Ascoli Theorem.

Finally, item  \ref{it:6} follows from item \ref{it:5} and the fact that every weakly convergent sequence is bounded.
\end{proof}

The following space plays an important role in our results.

\begin{definition}\label{def:Pi} Let $\phi$ be a $G$-function. We define

\[\Pi(\phi, r):=\{u \in L^\phi \mid d(u,L^\infty)<r\}.\]

\end{definition}

It is well known that if $\phi$ is a $\Delta_2$ function then $\Pi(\phi, r)= L^\phi$.

 \section{Differentiability of action integrals}

 Let $X$ be a separable Banach space and $\mathcal{L}:[0,T]\times X\times X \rightarrow \mathbb{R}$, a function which satisfies the \textit{ Caratheodory condition}, i.e.

\begin{itemize}\label{caratherodory}
    \item $\mathcal{L}(\cdot,x,y)$ measurable on $[0,T]$ for every $(x,y)$ in $X\times X$.
    \item $\mathcal{L}(t,\cdot,\cdot)$ is continuous on $X\times X$ a.e. $t\in [0,T]$.
\end{itemize}

The Caratheodory condition implies that the map $t\to \mathcal{L}(t,u_1(t),u_2(t))$ is measurable when $u_1,u_2$ are measurable functions with  respect $t$ (see \cite[Sección 6.2]{papageorgiou2009handbook}).

In this section, we study the generalized Clarke subdifferentiability (see \cite{papageorgiou2009handbook,clarke2013functional,clarke1990optimization}) of integral functionals of the following type

\begin{equation}\label{integraldeacción}
  \mathcal{I}(u):= \int_0^T \mathcal{L}(t,u(t),u'(t))  dt,  \tag{IA}
\end{equation}
 on Orlicz–Sobolev spaces.

In what follows we assume de following hypotheses

\begin{enumerate}

 \labitem{H10}{hip1dif} $\phi:X \rightarrow [0,\infty)$ is a  $G$-function  bounded on bounded subsets of $X$.

  \labitem{H11}{hip2dif} $\mathcal{L}(t,\cdot,\cdot):X\times X \rightarrow \mathbb{R}$ is locally Lipschitz and regular for a.e. $t \in [0,T]$.

 \labitem{H12}{hip3dif}

 There exists $\lambda, \Lambda > 0$, $b \in L^1([0,T];[0,\infty))$ and $a:X \rightarrow [0,+\infty) $ bounded on bounded subsets such that for a.e. $t \in [0,T]$, if
$ (x,y) \in X\times X$ and $\xi_t=(\xi_{xt},\xi_{yt}) \in \partial_{x}\mathcal{L}(t,x,y)\times\partial_{y}\mathcal{L}(t,x,y)$ then
      \begin{equation}\label{Hip4}
          |\mathcal{L}(t,x,y)| + \|\xi_{xt}\|_{X^\star}+\phi^\star \left(\frac{\xi_{yt}}{\lambda}\right) \leq a(x) \left(\phi \left(\frac{y}{\Lambda}\right)+ b(t)\right)
     \end{equation}
\end{enumerate}

We start proving some auxiliary results. 

\begin{lem}\label{lema:cotauniforme}
Given $u_0 \in \Pi(\phi, \Lambda),$ there exists $ \delta=\delta(u_0,\Lambda)>0$ and $ M=M(u_0,\Lambda)$ such that
\[\int_0^T{\phi \left(\frac{v(t)}{\Lambda}\right) dt}< M,\] 
for every  $ v \in B(u_0, \delta)$.
\end{lem}

\begin{proof} 
Let  $u_0 \in \Pi(\phi, \Lambda )$. We can find $v_0 \in L^\infty$ with $d(u_0, v_0) < \Lambda$. We take $ v \in B(u_0, \delta)$, where $\delta > 0$ and $\delta< \Lambda-\|u_0-v_0\|_{L^\phi}$.
We choose  $ \epsilon>0$ small enough in order to $\|u_0-v_0\|_{L^\phi} \leq \left(\Lambda-\delta-\epsilon\Lambda\right)$. Then, from the convexity of $\phi$
\begin{multline*}
  \phi \left(\frac{v(t)}{\Lambda} \right)  %\phi \left(\frac{v(t)-u_0(t)+u_0(t)-v_0(t)+v_0(t)}{\Lambda} \right) \\
%
 % = \phi \left(\frac{(v(t)-u_0(t)) \|v-u_0\|_{L^\phi}}{\Lambda \|v-u_0\|_{L^\phi}} + \frac{(u_0(t)-v_0(t))\left(1-\frac{\delta}{\Lambda}-\epsilon\right)}%%%{\Lambda\left(1-\frac{\delta}{\Lambda}-\epsilon\right)} + \frac{\epsilon v_0(t)}{\epsilon \Lambda}\right) \\
%
  \leq \frac{\delta}{\Lambda}\phi \left(\frac{v(t)-u_0(t)}{\delta} \right)\\
  +\left(1-\frac{\delta}{\Lambda}-\epsilon\right) \phi \left(\frac{u_0(t)-v_0(t)}{\Lambda-\delta-\epsilon\Lambda}\right) + \epsilon \phi \left( \frac{v_0(t)}{\epsilon\Lambda }\right).\
\end{multline*}
Integrating this inequality and using $\|v-u_0\|_{L^\phi}<\delta$, $\|u_0-v_0\|_{L^\phi} \leq \left(\Lambda-\delta-\epsilon\Lambda\right)$ and   Proposition \ref{propos:Gfuncion_elementales}(2)   we obtain

\[\int_0^T\phi \left(\frac{v(t)}{\Lambda} \right) dt\leq 1+\epsilon \int_0^T\phi \left( \frac{v_0(t)}{\epsilon\Lambda }\right)dt=:M\]

Since $\phi$ is bounded on bounded sets,  we have that $M<\infty$.   Moreover, by construction, $v_0$ is independent of $v$.
\end{proof}
The following result follows easily from the definitions.
\begin{lem}\label{lema:cotanemi}
 If $a:X\to[0,\infty)$ is a function which is bounded on bounded subsets of $X$,  then  there exists a non decreasing function  $K:[0,\infty)\to[0,\infty)$ such that $\|a(u)\|_{L^\infty}\leq K(\|u\|_{L^\infty})$
\end{lem}

\begin{Propos} The funcional
\[ I(u_1,u_2):= \int_0^T{\mathcal{L}(t,u_1(t),u_2(t)) dt}\] is locally Lipschitz on $L^\infty([0,T],X)\times\Pi(\phi, \Lambda)$.
\end{Propos}
\begin{proof}
From condition \ref{hip3dif}, Lemma \ref{lema:cotanemi} and using \cite[Th. 5.1]{orliczvectorial2005} we have that 
\begin{multline*}
 |\mathcal{L}(t,u_1(t),u_2(t))| \leq a(u_1(t)) \left(\phi \left(\frac{u_2(t)}{\Lambda}\right)+ b(t)\right)\\
 \leq K(\|u_1\|_{L^\infty})\left(\phi \left(\frac{u_2(t)}{\Lambda}\right)+ b(t)\right) \in L^1([0,T],\mathbb{R}).
\end{multline*}
Hence $I(u)<\infty$.

Let $u=(u_1,u_2), v=(v_1,v_2) \in L^\infty([0,T],X)\times L^\phi ([0,T],X)$. Then by \ref{hip2dif}, Mean Value Theorem \cite[Th. 10.17]{clarke2013functional}, there exists $\xi_t \in \partial \mathcal{L}(t,w(t))$ such that
\[ \mathcal{L}(t,u(t))-\mathcal{L}(t,v(t)) = \left< \xi_t ; u(t)-v(t)\right>.\]
 The function  $w(t)=(w_1(t),w_2(t))$ belongs to the line segment joining $u(t)$ and $v(t)$, i.e.  $w(t)=\alpha(t)u(t)+(1-\alpha(t))v(t)$, where $\alpha(t) \in [0,1]$.
Since  $\mathcal{L}(t,\cdot,\cdot)$ is regular,  from \cite[Prop. 2.3.15]{clarke1990optimization}, we can identify  $\xi_t$ with an element $(\xi_{xt},\xi_{yt}) \in \partial_{x} \mathcal{L}(t,w(t))\times \partial_{y} \mathcal{L}(t,w(t))  $.

As consequence
\begin{equation}\label{ecua:acotac}
    \begin{split}
	\mathcal{L}(t,u(t))-\mathcal{L}(t,v(t)) &= \left< \xi_{xt}; u_1(t)-v_1(t)\right> +\left< \xi_{yt}; u_2(t)-v_2(t)\right>\\
	&=:I_1+I_2.
	\end{split}
\end{equation}
From \ref{hip3dif} and by Lemma \ref{lema:cotanemi}, we deduce that, for almost every where $ t\in[0,T]$,
\begin{equation}\label{ecua:sum1}
 \begin{split}
   I_1  & \leq  \|\xi_{xt}\|_{X^\star} \|u_1(t)-v_1(t)\|_X \\
&\leq  a(w_1(t)) \left [\alpha(t) \phi\left( \frac{u_2(t)}{\Lambda} \right)+(1-\alpha(t))\phi\left( \frac{v_2(t)}{\Lambda}\right)+b(t)\right]\\
&\quad  \times \|u_1-v_1\|_{L^\infty}\\
 &\leq  K(\|w_1\|_{L^\infty}) \left [\phi\left( \frac{u_2(t)}{\Lambda} \right)+\phi\left( \frac{v_2(t)}{\Lambda} \right) + b(t)\right]\|u_1-v_1\|_{L^\infty }.
 \end{split}
\end{equation}

On the other hand, from Proposition \ref{propos:Gfuncion_elementales} (4)
\begin{equation}\label{provisorio}
I_2\leq
  |\lambda| \|u_2-v_2\|_{L^\phi}  \left[\phi^\star \left( \frac{\xi_{yt}}{\lambda}\right) +\phi \left( \frac{u_2(t)-v_2(t)}{\|u_2-v_2\|_{L^\phi}}\right) \right].
\end{equation}

In a similar way to  the inequalities \eqref{ecua:sum1}
\begin{equation}\label{provisorio2}
\phi^\star \left( \frac{\xi_{yt}}{\lambda}\right)\leq K(\|w_1\|_{L^\infty}) \left [\phi\left( \frac{u_2(t)}{\Lambda} \right)+\phi\left( \frac{v_2(t)}{\Lambda} \right) + b(t)\right].
\end{equation}

Taking account of \eqref{provisorio} y \eqref{provisorio2}, and writing $z=\frac{u_2(t)-v_2(t)}{\|u_2-v_2\|_{L^\phi}}$,
\begin{equation}\label{ecua:sum2}
I_2\leq | \lambda| \|u_2-v_2\|_{L^\phi}
\left\{K(\|w_1\|_{L^\infty}) \left [\phi\left( \frac{u_2}{\Lambda} \right)+\phi\left( \frac{v_2}{\Lambda} \right) + b(t)\right]  + \phi\left( z\right) \right\}
\end{equation}
Finally from \eqref{ecua:acotac}, \eqref{ecua:sum1}, \eqref{ecua:sum2}, if $u_0 \in L^\infty([0,T],X)\times\Pi(\phi, \Lambda)$ and we take $\delta$ as in Lema \ref{lema:cotauniforme}, we arrived at
\[
|I(u)-I(v)|
\leq K(\|u_1-v_1\|_{L^\infty}+\|u_2-v_2\|_{L^\phi}), \hspace{0,3cm} \forall  u,v \in  B(u_0, \delta).
\]
\end{proof}
As consequence the previous theorem, we can apply \cite[Def. 10.3]{clarke2013functional} we can assure $\partial I(u)\neq \emptyset$. The following theorem give us a characterization of the set $\partial I(u)$. 

\begin{teo}\label{teodiff}
We assume \ref{hip1dif}-\ref{hip3dif}. Then we have that:
\begin{equation}\label{inclusionintegral}
    \partial I(u) \subset \int_0^T{\partial \mathcal{L}(t,u(t)) dt,}
\end{equation} 
for every $u=(u_1,u_2) \in L^\infty([0,T],X) \times \Pi(\phi, \Lambda)$.
\end{teo}
\begin{remark}
 The inclusion \eqref{inclusionintegral} is interpreted as follows, if $\xi \in \partial I(u)$, for \text{ a.e. } $t\in [0,T]$, there exists $\xi_t \in \partial \mathcal{L}(t,u(t))$  such that if $v=(v_1,v_2)\in L^\infty([0,T],X)\times L^\phi([0,T],X) $, the map $t \to \left<\xi_t, v(t)\right> \in L^1([0,T],\mathbb{R})$ and
\[\left<\xi, v\right>=\int_0^T \left<\xi_t , v(t)\right> dt.\]

\end{remark}

\begin{proof}
First, let us see that

\begin{equation}\label{ineteodif5}
 \int_0^T{\mathcal{L}^0(t,u(t);v(t))dt} \geq I^0(u; v),
\end{equation}
for every $v  \in L^\infty([0,T],X)\times L^\phi([0,T],X).$
We know that 
\[ I^0(u; v)= \limsup\limits_{y\rightarrow u, \lambda \downarrow 0} \int_0^T \frac{ \mathcal{L} (t, y(t)+ \lambda v(t))- \mathcal{L}(t,y(t))}{\lambda} dt.\]
We can find $\{y_n\}_{n\in \mathbb{N}},\{\lambda_n\}_{n\in \mathbb{N}}$
with $y_n\to u$, $\lambda_n \downarrow 0$ . 

\begin{equation}\label{limsup}
   I^0(u; v)= \lim_{n\rightarrow \infty} \int_0^T \frac{ \mathcal{L} (t,y_n(t)+ \lambda_n v(t))- \mathcal{L}(t,y_n(t))}{\lambda_n} dt.
\end{equation}

We write $y_n=(y_{1,n},y_{2,n})$. Since $u_2\in \Pi(\phi, \Lambda)$, there exists $0<\alpha<1$ small enough such that $d(u_2, L^\infty)<(1-\alpha)\Lambda$. From \cite[Lemma 4.7]{Mazzone2019} there exists a subsequence of $\{y_{2,n}\}$(w.l.g. again denoted  $\{y_{2,n}\}$) and $h \in L^1([0,T], \mathbb{R})$, satisfying $y_{2,n} \to u_2$ a.e. $t\in [0,T]$ and
\begin{equation}\label{ineaux1}
 \phi \left(\frac{y_{2n}}{(1-\alpha)\Lambda}\right)\leq h.
\end{equation}

On the other hand, we know that $v_2 \in L^\phi([0,T],X)$. Hence, there exists $\delta>0$, satisfying
\begin{equation}\label{ineaux2}
   \int_0^T \phi(\delta v_2)dt <\infty.
\end{equation}

Now,  using the Mean Value Theorem \cite[Th. 10.17]{clarke2013functional}, we conclude that there exists $\xi_{tn} \in \partial \mathcal{L}(t,w_n(t))$, where
 $w_n(t)$ lies on the line segment joining $y_n(t)+\lambda_n v(t)$ and $y_n(t)$,  such that

 \begin{equation}\label{ineteodif4}
  \frac{\mathcal{L}(t,y_n(t)+\lambda_n v(t))-\mathcal{L}(t,y_n(t))}{\lambda_n} = \left< \xi_{tn} ; v(t)\right>.
 \end{equation}
We write $\xi_{tn}:=(\zeta_{tn},\omega_{tn})$ (see  \cite[Prop. 2.3.15]{clarke1990optimization} ), so that
\begin{equation}\label{ineteodif0}
 \left< \xi_{tn}, v(t)\right>=\left<\zeta_{tn},v_1(t)\right>+\left<\omega_{tn},v_2(t)\right>.
\end{equation}

By \ref{hip3dif} and following the same reasoning as  in inequality \eqref{ecua:sum1}, we obtain

\begin{multline}\label{ineteodif1}
\left< \zeta_{tn}, v_1(t)\right>  \\
\leq a(w_{1n})\left((1-\alpha)\phi\left( \frac{y_{2,n}}{(1-\alpha)\Lambda}\right)+\alpha\phi\left( \frac{\lambda_n v_2}{\alpha\Lambda} \right)+\phi\left(\frac{y_{2,n}}{\Lambda}\right)+b(t)\right)\|v_1\|_{L^\infty}.
\end{multline}

In similar way as in inequalities \eqref{provisorio}-\eqref{ecua:sum2}

\begin{multline}\label{ineteodif2}\small
\left< \omega_{tn},v_2(t)\right>
 \leq |\lambda| \|v_2\|_{L^\phi}\\
 \times \left\{ a(w_{1n}) \left [(1-\alpha)\phi\left( \frac{y_{2,n}}{(1-\alpha)\Lambda}\right)+\alpha\phi\left( \frac{\lambda_n v_2}{\alpha\Lambda}\right)+\phi\left( \frac{y_{2,n}}{\Lambda} \right) + b\right]\right. \\
 \left.+ \phi\left( \frac{v_2}{\|v_2\|_{L^\phi}}\right) \right\}.
 \end{multline}

Since $y_{1n}+\lambda_n v_1 \to u_1$ in $L^\infty([0,T],X)$, we can assume
  \[a(w_{1n})\leq K(\|y_{1n}+\lambda_n v_1\|_{L^\infty}+\|y_{1n}\|_{L^\infty})\leq K(\|u\|_{L^\infty}+1).\]
  
Taking $n$ large enough so that ${\lambda_n}/{\alpha \Lambda}<\delta$, and using \eqref{ineaux1}, \eqref{ineaux2},\eqref{ineteodif2} and the Proposition \ref{propos:Gfuncion_elementales}(2),  we conclude that there exists $h_1 \in L^1([0,T],\mathbb{R})$ such as
$$\big|\left(\mathcal{L}(t,y_n(t)+\lambda_n v(t))-\mathcal{L}(t,y_n(t))\right)\lambda_n^{-1}\big|\leq h_1(t).$$
Finally, by Lebesgue dominated convergence theorem, we obtain
\begin{equation*}
\begin{split}
 \int_0^T{\mathcal{L}^0 (t,u(t),v(t)) dt}&\geq\int_0^T {\lim\limits_{n\rightarrow \infty}\frac{ \mathcal{L}(t,y_n+ \lambda_n v)- \mathcal{L}(t,y_n)}{\lambda_n} dt}\\
 &=\lim\limits_{n\rightarrow \infty} \int_0^T \frac{ \mathcal{L}(t,y_n+ \lambda_n v)- \mathcal{L}(t,y_n)}{\lambda_n} dt = I^0 (u,v).
\end{split}
\end{equation*}
Hence the proof of \eqref{ineteodif5} is completed.

Given 
$u=(u_1,u_2) \in L^\infty([0,T],X)\times\Pi(\phi, \Lambda)$, we define $\mathcal{L}_u:[0,T]\times X\times X\rightarrow \mathbb{R}$ in the following way,
\begin{equation}\label{equateodif1}
 \mathcal{L}_u(t,v) = \mathcal{L}^0(t,u(t);v), \quad v\in X\times X.
\end{equation}

We considerer the corresponding integral function of the function $\mathcal{L}_u$, given by 
\begin{equation}\label{equateodif2}
I_u(v) = \int_0^T{ \mathcal{L}_u(t,v(t))dt}. 
\end{equation}
for every $ v=(v_1,v_2)\in L^\infty([0,T],X)\times L^\phi([0,T],X).$
 
Let $\xi \in \partial I(u)$, by \eqref{ineteodif5}, \eqref{equateodif1} and \eqref{equateodif2} 
\[I_u(v)-I_u(0) =I_u(v)\geq \left<\xi, v\right>.\]

Because $I_u(\cdot)$ is a convex function, the previous inequality implies that $\xi \in \partial_{\text{conv}} I_u(0)$, where $\partial_{\text{conv}}$ denotes the subdifferential in the convex analysis sense (see \cite{clarke2013functional}). Then, we can  apply the results in \cite{ioffe1972subdifferentials} to the functional $I_u$ and conclude that $\exists \hspace{0,1cm}\xi_t \in \partial_{\text{conv}} \mathcal{L}_u(t,0) $ such that
\[\left<\xi, v\right> = \int_0^T{ \left<\xi_t, v(t)\right> dt}.\]

In a similar way, since $\mathcal{L}_u(t,\cdot)$ is convex, we have
\[\left<\xi_t;v\right>\leq \mathcal{L}_u(t;v)-\mathcal{L}_u(t;0)=\mathcal{L}^0(t,u;v).\]
Hence, $\xi_t \in \partial \mathcal{L}(t,u(t))$.
\end{proof}

We define the linear operator $\mathbb{F}:W ^{1,\phi} \to  L^\infty  \times L^\phi ,$
by $\mathbb{F}(u):=(u,u')$. Therefore $\mathcal{I}=I \circ \mathbb{F}$.  Using Proposition \ref{prop:prop_espacios} \eqref{it:1}   we obtain that $\mathbb{F}$ is bounded. In particular $\mathbb{F}$ has a Fréchet derivative, in fact $\mathbb{F}'(u)=\mathbb{F}$. So, if we apply the chain rule \cite[Th. 10.19]{clarke2013functional}
\[\partial \mathcal{I}(u)\subset \mathbb{F}'^\star \partial I(\mathbb{F}(u))\]
where, as is usual, $\mathbb{F}'^\star$ denotes the adjoint operator, i.e $\mathbb{F}'^\star:(L^\infty\times L^\phi)^\star \to (W ^{1,\phi})^\star$, and it satisfies for every $\xi \in (L^\infty\times L^\phi)^\star \text{ and } v \in W ^{1,\phi}$ that
$   \left<\mathbb{F}^\star\xi, v\right>=\left<\xi, \mathbb{F} (v)\right>
$.
As consecuence if $\zeta \in \partial \mathcal{I}(u)\subset W^{-1,\phi}$, there exists $\xi \in \partial I(\mathbb{F}(u)) $ with $ \zeta=\mathbb{F}^\star\xi $. Since $\xi \in \partial I(\mathbb{F}(u))$, from Theorem \ref{teodiff} exists $\xi_t=(\xi_{xt},\xi_{yt}) \in \partial_{x} \mathcal{L}(t,\mathbb{F}(u(t)))\times \partial_{y} \mathcal{L}(t,\mathbb{F}(u(t))) $ such that every $v \in L^\infty\times L^\phi$ we have
\begin{equation}\label{adjunto2}
   \left<\xi, v\right>=\int_0^T \left<\xi_t , v(t)\right> dt. 
\end{equation}
Now, if $v \in W ^{1,\phi}$, applying \eqref{adjunto2} we obtain that
\begin{equation}\label{adjunto1}
  \left<\zeta, v\right>=\left<\mathbb{F}^\star\xi, v\right>=\left<\xi, (v,v')\right>=\int_0^T \left<\xi_{xt} , v(t)\right>+\left<\xi_{yt} , v'(t)\right> dt.  
\end{equation}

The following corollary is thus proved.
\begin{Cor}\label{inclusionderivadafuncional} Suppouse that $u \in \wphi([0,T], X)$ satisfies $ u' \in \Pi(\phi, \Lambda)$. Asuming the condition \ref{hip1dif}-\ref{hip3dif} we have that \[\partial \mathcal{I}(u)\subset \int_0^T \partial \mathcal{L}(t,u(t),u'(t))dt.\]
\end{Cor}

\section{Proofs of Main Theorems}
Henceforth, we assume that $X=\mathbb{R}^n$ i.e. we will consider Sobolev Spaces of $\mathbb{R}^n$-valued functions. We now introduce the following Lagrangian function:
\begin{equation}
\mathcal{L}(t,x,y)=\phi(y)+F(t,x),
\end{equation}
where we assume:

\begin{itemize}
 \item $\phi:\mathbb{R}^n \to \mathbb{R}$  is a continuously differentiable $G$-function,
 \item $F$ is a locally Lipschitz and regular function for almost everywhere $t\in [0,T]$,
 \item $F$ satisfies \ref{condacotación}.
\end{itemize}

From now on, we will omit the domain and codomain in the notation of functional spaces—for example, instead of writing $ L^1([0,T], \mathbb{R}^n) $, we will simply write $ L^1 $.

\begin{Propos}\label{ProposAcotación}
The Lagrangian $\mathcal{L}$ satisfies  \ref{hip1dif}-\ref{hip3dif}.
\end{Propos}
\begin{proof}

 The condition \ref{hip2dif} follows from the assumptions on $F$ and the results in \cite[Sec. 10.2]{clarke2013functional}.

 From Proposition \ref{propos:Gfuncion_elementales} (3), we know that  $\phi^\star(\nabla \phi(y))\leq \phi(2y)$, which, together with \ref{condacotación}, implies \ref{hip3dif} with $\lambda=1, \Lambda=1/2$ y $a(x)=\max \{c(x),2\}$. \end{proof}

\begin{Propos}\label{teosolucion}
Let $\mathcal{I}:\wphit \to \overline{\mathbb{R}}$  be the functional defined in \eqref{funcional}. We suppose that $\phi$ is a strictly convex function. Then, if $u_0 \in \wphit$ is a critical point   of $\mathcal{I}$ ($0\in\partial\mathcal{I}(u_0)$) and $u'_0\in \Pi(\phi,\Lambda)$,  then $u_0$ solves problem \eqref{eq_1}.
\end{Propos}
\begin{proof}
According to Proposition \ref{ProposAcotación}, Corollary \ref{inclusionderivadafuncional} and \cite[Th. 10.13]{clarke2013functional} we deduce that if $\zeta \in \partial\mathcal{I}(u)$, $\exists \hspace{0,1cm} \xi_{t} \in \partial F(t,u(t))$ such that
\[\left<\zeta,v\right> =\int_0^T \nabla \phi(u'(t))\cdot v'(t)dt+\int_0^T\xi_{t} \cdot v(t) dt \qquad \forall \hspace{0,1cm} v \in \wphit.\]
Therefore, if $u_0$ is critical point of $\mathcal{I}$ we get
\[\int_0^T \nabla \phi(u'_0(t))\cdot v'(t)dt=-\int_0^T\xi_{t} \cdot v(t) dt \qquad \forall \hspace{0,1cm} v \in \wphit,\]
By Fundamental Lemma of Calculus of Variations (see reference \cite{mawhin2010critical}) we infer that $\exists \hspace{0,1cm} C\in \mathbb{R}^n$ such that
\begin{equation}\label{consecuencialemafundamental}
 \nabla \phi(u'_0(s))=\int_0^s \xi_t dt + C \qquad\text{and} \qquad \int_0^T \xi_t dt = 0.
\end{equation}
Hence $\nabla \phi(u'_0)$ is an absolutelly continuous function and
\[\frac{d\nabla \phi(u'_0(t))}{dt}=\xi_t \in \partial F(t,u_0(t)) \qquad \text{a. e. } t\in [0,T].\]

Therefore, the inclusion in problem \eqref{eq_1} holds.

The strict convexity of $\phi$, implies that $\nabla \phi$ is a one-to-one map. Since \eqref{consecuencialemafundamental} implies that $\nabla \phi(u'_0(0))=\nabla \phi(u'_0(T))=C$, we conclude that $u'_0(0)=u'_0(T)$. The boundary condition $u_0(0)=u_0(T)$ is automatically satisfied by functions in the space $\wphit$.
\end{proof}

For the proof of Theorem~\ref{coercotasubgrad}, we need to use \cite[Th.~3.6]{buttazzo1998one}. We briefly recall that this theorem establishes the sequential lower semicontinuity of the functional
\[
I(u,v) = \int_0^T \phi(v) + F(t,u) \, dt
\]
in the space $ L^1 \times L^1 $, endowed with the product topology, where the strong topology is considered on the first factor and the weak topology on the second.

\vphantom{a}\\

\noindent \textbf{Proof of Theorem \ref{coercotasubgrad}}
For fixed $ x, y \in \mathbb{R}^n $ and $ t $ such that $ F(t, \cdot) $ is locally Lipschitz, we define
$
g_t(s) = F(t, x + s(y - x)).
$
We observe that $ g_t $ is a Lipschitz function, because $ F $ is locally Lipschitz with respect to its second variable, and is therefore globally Lipschitz on compact subsets.
As a result, $ g_t $ is absolutely continuous, so there exists $ g_t'(s) $ for almost every $ s \in [0,1] $, and

\begin{equation}\label{formaintegral}
 F(t,y)-F(t,x)=g_t(1)-g_t(0)=\int_0^1g_t'(s)ds=\int_0^1 \left<\xi_s,y-x\right>ds \quad \text{ a.e. } t,
\end{equation}
where in the last identity we applied the chain rule and $\xi_s \in \partial F(t,x+s(y-x))$.
Let $\lambda$ be a positive constant whose exact value will be determined later. Because  $\phi_0 \llcurly \phi$, there exists $C(\lambda)>0$ such that
\begin{equation}\label{desigualdadphicero}
 \phi_0(x)\leq \phi \left(\frac{x}{2\lambda}\right)+C(\lambda), \quad x\in \mathbb{R}^n.
\end{equation}
We write $u=\bar{u}+\tilde{u}$ (see Section \ref{Preliminares}). Then,
using \eqref{formaintegral}, Fenchel inequality,  \ref{condB}, the convexidad of $\phi_0$, \eqref{desigualdadphicero}
and Lemma \ref{lem:inclusion orliczII}, we obtain

\begin{eqnarray*}
J&:=&\left|\int_0^T F(t,u)-F(t,\bar{u})dt\right|\leq\int_0^T|F(t,u)-F(t,\bar{u})|dt\\
&=&\int_0^T\int_0^1 \left|\left<\xi_s,\tilde{u}\right>\right|ds dt\\
&\leq& \lambda \int_0^T d(t)\int_0^1 \left[\phi^\star\left(\frac{\xi_s}{d(t)}\right)+\phi \left(\frac{\tilde{u}}{\lambda}\right)\right]ds dt\\
&\leq& \lambda \int_0^T d(t)\int_0^1\left[\phi_0(\bar{u}+s\tilde{u})+1+\phi \left(\frac{\tilde{u}}{\lambda}\right)\right]ds dt\\
&\leq& \lambda \int_0^Td(t)\int_0^1\left[ \frac{1}{2}\phi_0(2\bar{u})+\frac{1}{2}\phi_0(2\tilde{u})+\phi \left(\frac{\tilde{u}}{\lambda}\right)+1\right]ds dt\\
&\leq& \lambda \int_0^Td(t)\int_0^1\left[ \phi_0(2\bar{u})+\frac{3}{2}\phi \left(\frac{\tilde{u}}{\lambda}\right)+C\right]ds dt\\
&\leq& C \phi_0(2\bar{u})+\lambda C_1\int_0^T\phi\left(\frac{Tu'(s)}{\lambda}\right)ds +C
\end{eqnarray*}
where $C = C(\|d\|_{L^1}, \lambda)$ but $C_1 = C_1(\|d\|_{L^1})$. Since $\phi^\star \in \Delta_2$, we can take $l$ satisfying \eqref{eq:nabla2} with $r = \frac{1}{2} \min\{(C_1 T)^{-1}, 1\}$. We note that $l$ is independent of $\lambda$, therefore we can take  $\lambda = lT$. Hence, we obtain

\[J\leq C\phi_0(2\bar{u})+\frac{1}{2}\int_0^T\phi(u'(s))ds+C.\]
Then
\begin{eqnarray*}
\mathcal{I}(u)&=&\int_0^T\phi(u')+F(t,u)-F(t,\bar{u})+F(t,\bar{u})dt\\
&\geq&\frac{1}{2}\int_0^T\phi(u')dt-C \phi_0(2\bar{u})+\int_0^T F(t,\bar{u})dt-C.
\end{eqnarray*}
Suppose that $\|u_n\|_{W^{1,\phi}}\to\infty$. From Proposition \ref{prop:prop_espacios} item \ref{it:4} can assume that $|\bar{u}_n|\to \infty$ or that $\|u'_n\|_{L^\phi}\to \infty$ and $\bar{u}$ remains bounded. In the first case $\mathcal{I}(u_n)\to \infty$ from \ref{condcoermedia1}.  In the second case, \eqref{eq:amemiya}  implies $\int_0^T \phi(u'_n) dt +1\geq \|u'_n\|_{L^\phi} \to \infty.$ and therefore, also in this case $\mathcal{I}(u_n)\to \infty$. In conclusion, we have proved that $\mathcal{I}$ is a coercive functional.

Let $\{u_n\}_{n\in \mathbb{N}}$ be a minimizing sequence, i.e.
\[
\mathcal{I}(u_n) \to \inf_{u} \mathcal{I}(u).
\]
In particular, $\mathcal{I}(u_n)$ is a bounded sequence in $\mathbb{R}$. Hence, since $\mathcal{I}$ is coercive, the sequence $\{u_n\}$ is bounded in $\wphit$. By Proposition \ref{prop:prop_espacios}, item \ref{it:5}, there exists a subsequence of $\{u_n\}$ (again denoted by $\{u_n\}$) which converges uniformly to a certain continuous function $u_0$ with $u_0(0)=u_0(T)$.

From \cite[Th. 7.2]{orliczvectorial2005}, $L^{\phi}$ is reflexive. Since $u'_n$ is a bounded sequence in $L^\phi$, then by \cite[Th. 3.18]{brezis2011functional}, there exists a subsequence (still denoted by $u_n$) such that $u'_n$ converges weakly to $v$ in $L^{\phi}$. It is easy to see that $v$ is the weak derivative of $u_0$, and hence $u_0 \in W^{1,\phi}_T$.

As $L^\infty \hookrightarrow L^{\phi^\star}$, we have that $u'_n$ converges to $u'_0$ in the weak topology of $L^1$. Hence, the hypotheses of \cite[Th. 3.6]{buttazzo1998one} are satisfied, and we conclude that
\[
\mathcal{I}(u_0) \leq \liminf_{n \to \infty} \mathcal{I}(u_n) = \inf_{u} \mathcal{I}(u).
\]
Consequently, $u_0$ is a minimizer of $\mathcal{I}$, and thus a critical point. Using Proposition \ref{teosolucion}, the proof of the theorem is complete. \qed

\begin{remark}
 We note that the Theorem \ref{coercotasubgrad} remains true when condition \ref{condcoermedia1} is replaced by the  weaker condition \[\liminf_{|x|\to\infty}\frac{1}{\phi_0(2x)}\int_0^T F(t,x)dt>C,\]
 where $C$ is a suitable constant $C$.
\end{remark}

\noindent\textbf{Proof Theorem \ref{coercuasi}.} Let $u \in W^{1,\phi}$. We will use, once again, the decomposition
$u=\overline{u}+\tilde{u}$.  Taking into account \ref{cuasisub}, we get
\begin{equation}\label{desigualdadpotencial}
 \frac{1}{\mu}F(t,\lambda\overline{u})-F(t,-\tilde{u})\leq F(t,u).
\end{equation}
We take any $l>\frac{1}{T}\int_{0}^{T}b(t)dt+1$. Applying   \eqref{eq:orden},   with  $ k=\frac{1}{Tl}$, and using Lemma \ref{lem:inclusion orliczII}, we get $C=C(k)>0$ such that

\begin{eqnarray*}
F(t,-\tilde{u})&\leq& b(t)(\phi_0(\tilde{u})+1)\leq b(t)\left(C+\phi\left(\frac{\tilde{u}}{Tl}\right)+1\right)\\
&\leq&b(t)\left(C+\frac{1}{T}\int_0^T\phi\left(\frac{u'}{l}\right)dt+1\right)\\
&=& b(t)C+\frac{b(t)}{lT}\int_0^T\phi\left(u'\right)dt.
\end{eqnarray*}
Then
\begin{equation}\label{desigualdadintegralpotencial}
\int_0^T F(t,-\tilde{u})dt\leq C+R\int_0^T\phi(u'(t))dt,
\end{equation}
where $R= \frac{1}{lT}\int_{0}^{T}b(t)dt<1$.
Using \eqref{desigualdadpotencial}, \eqref{desigualdadintegralpotencial}
\begin{eqnarray*}
\mathcal{I}(u)
&\geq& \int_0^T\phi(u'(t))dt+\frac{1}{\mu} \int_0^T F(t,\lambda\overline{u})dt-\int_0^T F(t,-\tilde{u})dt\\
&\geq&\int_0^T \phi(u'(t))dt + \frac{1}{\mu} \int_0^T F(t,\lambda\overline{u})dt - C-R\int_0^T\phi(u'(t))dt\\
&=& (1-R)\int_0^T\phi(u'(t))dt+ \frac{1}{\mu} \int_0^T F(t,\lambda\overline{u})dt -C.
\end{eqnarray*}

The remainder of the proof follows the same steps as in the proof of  Theorem \ref{coercotasubgrad}.\qed

For the proof of Theorem \ref{coerBO}, we will need to use the so-called \emph{Matuszewska-Orlicz indices}. We will recall bellow their definition and  some basic properties of them. For further details on the topic, we recommend the reference \cite{m}.

\begin{definition}\label{definicionoindices}
Let $\phi$ be an $N$-function. We define the following limits, which are known as the \emph{Matuszewska-Orlicz indices}\index{Matuszewska-Orlicz indices}.

\[
\alpha_{\phi} = \lim_{\lambda \rightarrow 0^+} \frac{\ln\left(\sup_{u \neq 0} \frac{\phi(\lambda u)}{\phi(u)}\right)}{\ln(\lambda)} \text{ and } \beta_{\phi} = \lim_{\lambda \rightarrow \infty} \frac{\ln\left(\sup_{u \neq 0} \frac{\phi(\lambda u)}{\phi(u)}\right)}{\ln(\lambda)}.
\]

\end{definition}

\begin{Propos}\label{prop:prop_mat}
Let $\phi$ be an $N$-function. Then the following statements holds:
\begin{enumerate}\label{propiedadesindices}
\item $1\leq \alpha_{\phi}\leq \beta_{\phi}\leq \infty.$
\item $\frac{1}{\alpha_{\phi}}+\frac{1}{\beta_{\phi^\star}}=1.$
\item $\phi \in \Delta_2^G$ if and only if $\beta_{\phi}<\infty.$
\item $\phi^\star \in \Delta_2^G$ if and only if $\alpha_{\phi}>1.$
\end{enumerate}

\end{Propos}
\begin{proof} See \cite{m}.\end{proof}

The proof of next lemma can be  found in \cite[Lemma 5.2]{acinas2015some-2}.

\begin{lem}\label{lema5.2}
Let $\phi$ be a $G$-function. For all $0<\mu<\alpha_\phi<\infty$, we have
\[\lim_{\|u\|_{L^\phi}\to \infty}\frac{\rho_\phi(u)}{\|u\|_{L^\phi}^\mu}=\infty.\]
\end{lem}

\noindent\textbf{Proof Theorem \ref{coerBO}  }
Using Definition \ref{def:bo},  Hölder's inequality, and the Anisotropic  Poincaré-Wirtinger
Inequality (Proposition \ref{prop:prop_espacios}(2)), from which we obtain

\begin{eqnarray*}
\mathcal{I}(u) &=& \rho_\phi(u') + \int_0^T F(t,u)dt \\
&=& \rho_\phi(u') + \int_0^T F(t,u) - F(t,\bar{u})dt + \int_0^T F(t,\bar{u})dt \\
&\geq& \rho_\phi(u') - \int_0^T b(t)(1 + |\tilde{u}|)dt + \int_0^T F(t,\bar{u})dt \\
&\geq& \rho_\phi(u') - \|b\|_{L^1}(1 + \|\tilde{u}\|_{L^\infty}) + \int_0^T F(t,\bar{u})dt \\
&\geq& \rho_\phi(u') - \|b\|_{L^1}(1 + C\|u'\|_{L^\phi}) + \int_0^T F(t,\bar{u})dt,
\end{eqnarray*}
From now on, the reasoning continues as in the proof of the previous theorems using  and Lemma \ref{lema5.2} (with $\mu=1$).  \qed

\begin{remark}\label{obsumapotenciales}
 It is usual to find in the literature potentials $F$ that are the sum of two potentials $F=F_1+F_2$ that satisfy different conditions, see \cite{Tang2010,pacsca2007periodic} to mention just a few examples. We can use easily our results to contemplate such situations.  Note that the key to obtaining the existence of minima is to have coercivity for functional $\mathcal{I}$.
 We observe that if $F_i$, $i=1,2$ are  potentials  satisfying the hypotheses of any of Theorems \ref{coercuasi} or \ref{coercotasubgrad} or \ref{coerBO} then $2F_i$, $i=1,2$  also satisfy them.  Then, using that

 \begin{multline*}
   \mathcal{I}(u)=  \int_0^T\phi(u')+F(t,u)dt\\=\frac12\left(\int_0^T\phi(u')+2F_1(t,u)dt\right)+\frac12\left(\int_0^T\phi(u')+2F_2(t,u)dt  \right)=:\mathcal{I}_1(u)+\mathcal{I}_2(u)
 \end{multline*}
 we see that $\mathcal{I}$ is coercive too. Moreover, one could even weaken the assumptions about one of the potentials, say $F_1$, so that the corresponding functional $\mathcal{I}_1$ could be simply bounded from below.
\end{remark}

\section{Applications}\label{sec:applications}
In this section, we present several examples to illustrate how our results encompass, unify, and extend previously established ones.

\paragraph{$p$-Laplacian.}
When $\phi(x)= \frac{|x|^p}{p}$ with $1<p<\infty$, the differential operator \[\frac{d}{dt} \nabla\phi(u'(t))= \frac{d}{dt} \left\{|u'(t)|^{p-2}u'(t)\right\}\] is known as $p$-Laplacian. The problem of the existence of periodic solutions for nonlinear systems involving this operator has been extensively studied in the literature; see
\cite{pacsca2007periodic,lian2017periodic,li2015infinitely,zhang2012non,Tang2010,xu2007some,tian2007periodic,ma2009existence,Ning2015} for example.

Our results generalize these precedents in several directions. First, we consider a more general, inhomogeneous, and anisotropic differential operator. Second, even when our results are applied to the $p$-Laplacian operator, we obtain improved conditions on the potentials $F$.

For example, we note that condition (viii) in Theorem~$3$ of \cite{pacsca2007periodic}, namely
\begin{equation}\label{condpasca}
\text{if } \zeta \in \partial F(t,x) \text{ then } |\zeta| < c_1 |x|^\alpha + c_2,
\end{equation}
with $\alpha \in [0,p-1)$ is a particular case of our condition \ref{condB} (take $\phi_0(x)= |x|^{\alpha p/(p-1)}$ and $d$ is a suitable constant).
On the other hand, the following example from \cite{acinas2022} shows that there exists a potential $F$ that satisfies condition~\ref{condB} but does not satisfy \eqref{condpasca}.
  Let
\[
 \phi_1(x):=\int_0^{|x|}\frac{s^{p-1}}{\log (s^2+e)}ds.
\]
It was shown in \cite{acinas2022} that there exist $\phi_0\llcurly |x|^p/p$ where $\phi_0=\frac{|x|^p}{q[\log(|x|^2+e)]^q}$ for $x$ large enough and $q=p/(p-1)$. Morever it is verified that
\begin{equation}\label{ejemplohip4}
 \phi^\star\left(\phi_1'(x)\right)\leq \phi_0(x)+C,
\end{equation}
for some $C>0$.
We consider \[F(t,x)= \phi_1(x)(\sin(t)+1/2)+ g(t)\cdot x,\]
where $g:[0,2\pi] \to \mathbb{R}^n$ is an integrable function. We make a brief comment. For the purpose of presenting the mentioned example, it would have been sufficient to take $F=\phi_1$; the other terms in the expression are added solely to avoid trivial solutions and to ensure that the hypotheses are satisfied.

Taking $d=2(|g(t)|+1)$ and using \eqref{ejemplohip4} it is easy to verify  \ref{condB}.

We can think $F=F_1+F_2$ with $F_1= \phi_1(x)(\sin(t)+1/2)$ and $F_2=g(t)\cdot x $. The function $F_1$ satisfy  for some no negative constant $A$
\begin{eqnarray*}
 \frac{1}{\phi_0(2x)}\int_0^T F_1(t,x) dt&=&\frac{A \log(|x|^2+e)^{q-1}}{x^p}\int_0^x\frac{s^{p-1}\log(|x|^2+e)}{\log(s^2+e)}ds\\&>&\frac{A\log(|x|^2+e)^{q-1}}{x^p}\int_0^x s^{p-1}ds \to \infty \text{ when } |x|\to \infty.
\end{eqnarray*}

Additionally,
\[\frac{1}{\phi_0(2x)}\int_0^T F_2(t, x)dt=\frac{1}{\phi_0(2x)} \left(\int_0^T g(t) dt \right)\cdot x \to 0 \text { when }|x|\to \infty\] because $\phi_0$ is a $N$-function. Consequently,  conditions \ref{condcoermedia1} is satisfied by function $F$.  Then we can apply the Theorem \ref{coercotasubgrad}  and we obtain existence of periodic solutions for the $p$-Laplacian problem.

\begin{equation}\left\{
\begin{array}{ll}
   \frac{d}{dt} \left\{|u'(t)|^{p-2}u'(t)\right\} = \nabla_x F(t,u(t)) \quad \hbox{a.e.}\ t \in (0,T)\\
    u(0)-u(T)=u'(0)-u'(T)=0.
\end{array}\right.
\end{equation}

\paragraph{$(p,q)$-Laplacian.}
When $\phi(x_1,x_2)= \frac{|x_1|^p}{p}+\frac{|x_2|^q}{q}$ with $1<p,q<\infty$, the problem \eqref{eq_1} becomes the existence problem of periodic solution for $(p,q)$-Laplacian equation
$$
\begin{cases}\frac{d}{d t}\left(\left|u'_1(t)\right|^{p-2} u'_1(t)\right) \in \partial_{u_1} F\left(t, u_1(t), u_2(t)\right), & \text { a.e. } t \in[0, T], \\ \frac{d}{d t}\left(\left|u'_2(t)\right|^{q-2} u'_2(t)\right)\in \partial_{u_2} F\left(t, u_1(t), u_2(t)\right), & \text { a.e. } t \in[0, T], \\ u_1(0)-u_1(T)=u'_1(0)-u'_1(T)=0, & \\ u_2(0)-u_2(T)=u'_2(0)-u'_2(T)=0,\end{cases}
$$

In \cite{Pasca2011} D. Pa\c{s}ca obtained existence of solutions of the previous problem under the hypotheses

\begin{empheq}{align}
& \zeta_{1} \in \partial_{x_1} F\left(t, x_1, x_2\right) \Rightarrow\left|\zeta_{1}\right| \leq c_{11}\left|x_1\right|^{\alpha_1}+c_{12}\label{eq:pasca1} \\
& \zeta_{2} \in \partial_{x_2} F\left(t, x_1, x_2\right) \Rightarrow\left|\zeta_{2}\right| \leq c_{21}\left|x_2\right|^{\alpha_2}+c_{22}\label{eq:pasca2}
\end{empheq}
with $c_{11}, c_{12}, c_{21}, c_{22}\geq 0$ and $\alpha_1 \in[0, p-1), \alpha_2 \in[0, q-1)$
for all $t \in[0, T]$ and all $\left(x_1, x_2\right) \in \mathbb{R}^n \times \mathbb{R}^n$. Moreover, in \cite{Pasca2011}, it is assumed that

\begin{empheq}{equation}
\frac{1}{\left|x_1\right|^{p^{\prime} \alpha_1}+\left|x_2\right|^{q^{\prime} \alpha_2}} \int_0^T F\left(t, x_1, x_2\right) d t \rightarrow \infty\label{eq:pasca3}
\end{empheq}
where $p'$ ans $q'$ are the conjugate of $p$ and  $q$ respectively. It is easy to see that \eqref{eq:pasca1}--\eqref{eq:pasca3} imply  our conditions  \eqref{eq:condition_grad} and \eqref{eq:I->infinito}  with $\phi_0(x_1,x_2)=\left|x_1\right|^{p^{\prime} \alpha_1}+\left|x_2\right|^{q^{\prime} \alpha_2}$ and $d$ a positive number satisfying $2^{p'-1}(c_{1i}/d)^{p'}+2^{q'-1}(c_{2i}/d)^{q'}<1$, for $i=1,2$. The assumptions of Theorem 2.1 in \cite{Pasca2011} do not imply our inequality \eqref{eq:condacotación}. This inequality is a standard requirement that ensures that the functional $\mathcal{I}$  has a nonempty effective domain. We believe that the author may have omitted a necessary assumption to establish his results. Indeed, the function $ F = 1/t $ satisfies all the assumptions stated in his Theorem 2.1, but in this case  $ \mathcal{I\equiv\infty} $. It would suffice to add to Theorem 2.1 the simple requirement that $ F(t, 0) $ be an integrable function, since this, together with inequality (2.1)  in Theorem 2.1 of \cite{Pasca2011}, would  imply our inequality \eqref{eq:condition_grad}.

On the other hand, it is easy to check that if $p=q=3$ then  the functions $\phi_0(x_1,x_2)=|x_1|^2+|x_2|^2$, $F(x_1,x_2)=|x_1|^{13/6}+|x_2|^{13/6}+x_1x_2$  satisfy the hypotheses of  Theorem \ref{coercotasubgrad}   but not \eqref{eq:pasca1}  and \eqref{eq:pasca2}.

\section{Acknowledgements}
The research has been supported by SECyT-UNRC, grant C561 and FCEyN-UNLPam, grant PI 77 M.

\bibliographystyle{plain}
\bibliography{inclusion}

\end{document}